\documentclass[12pt]{article}
\usepackage{amsmath,amssymb,amsthm}
\usepackage{amsfonts, amsmath}
\newtheorem{theorem}{Theorem}[section]
\newtheorem{lemma}{Lemma}[section]
\newtheorem{corollary}{Corollary}[section]

\theoremstyle{definition}
\newtheorem{definition}{Definition}[section]
\theoremstyle{remark}
\newtheorem{remark}{Remarks}[section]

\numberwithin{equation}{section}

 \title{ What Mathematical Theories of Truth Should be Like\\
 (and Can be)}
\author {Seppo Heikkil\"a\\
Department of Mathematical Sciences, University of Oulu\\
BOX 3000, FIN-90014, Oulu, Finland\\
E-mail: sheikki@cc.oulu.fi}

\begin{document}
\maketitle 
\noindent
\begin{abstract} 
\noindent
Hannes Leitgeb formulated  eight norms for theories of truth in his paper: `What Theories of Truth Should be Like (but Cannot be)'.
We shall  present in this paper a mathematical theory of truth for suitably constructed languages which contain  the first-order language of set theory, and prove that it satisfies all those norms.
\vskip12pt

\vskip12pt


\end{abstract}

\baselineskip 16pt
\section{Introduction}\label{S1} 

Hannes Leitgeb formulated in \cite{[16]} the following norms for theories of truth: 
\begin{itemize}
\item [(a)] Truth should be expressed by a predicate (and a theory of syntax should be available).  
\item[(b)] If a theory of truth is added to mathematical or empirical theories, it
should be possible to prove the latter true.
\item[(c)] The truth predicate should not be subject to any type restrictions.
\item[(d)] $T$-biconditionals should be derivable unrestrictedly. 
\item[(e)]  Truth should be compositional.
\item[(f)] The theory should allow for standard interpretations. 
\item[(g)]  The outer logic and the inner logic should coincide.
\item[(h)] The outer logic should be classical. 
\end{itemize}

We present a mathematical theory of truth (shortly MTT) that satisfies norms (a)--(h). It is formulated for suitably constructed languages which contain the first-order language $L=\{\in\}$ of set theory. They are sublanguages of $\mathcal L=\{\in,T\}$, where $T$ is a monadic predicate. While not entailing all formalities of first-order languages, they  are closed under logical connectives and quantifiers of first--order predicate logic. 

  Zermelo-Fraenkel  (ZF) set theory, sets of its minimal model $M$, concepts definable in $L$  and  classical two-valued logic  are main tools. MTT  proves ZF set theory modelled by $M$ true. As stated in \cite{Ho}, ``First-order
Zermelo-Fraenkel set theory is widely regarded as the standard of
mathematical correctness, in the sense that a proof is correct if and only if it
can be formalised as a formal proof in set theory."

As shown in \cite{[BG]}, ``with some elementary assumptions about truth and logic, a logical disaster ensues."
It is caused by Liar paradox. 
Paradoxes led Zermelo to axiomatize set theory. By \cite{[16]} Tarski ``excluded all Liar-like sentences from being
well-formed" from his theory of truth (cf. \cite{[20]}).
Construction of sublanguages enables  MTT  to satisfy also the following norms:
\begin{itemize}
\item[(i)] The mathematical theory of truth should be free from paradoxes.
\item[(j)] Truth should be defined for languages within the very languages themselves.
\end{itemize}

MTT is  immune to `Tarski's Commandment' (cf.  \cite{[MG]}), to  Tarski's Undefinability Theorem  (cf. \cite{[20]}), to Tarskian hierarchies (cf. \cite{[H]}), and to `Revenge of the Liar' (cf. \cite{Fe}).      

Model $M$ can be replaced by any  countable  model of ZF  or
other set theory. Examples of such theories   are presented, e.g.,  in \cite{[Ku]}. 
The language $L$ can be replaced by the language of arithmetic, and $M$ by a standard model of arithmetic. In \cite{[9]} MTT is extended to a more general class of languages.


PKF theory (cf. \cite{[Ha],[8]}) axiomatizes Kripke's theory of truth (cf. \cite{[15]}). Both inner and outer logic of PKF are partial, whence norms (h) and (j) are not valid.   
Kripke's theory don't satisfy norms (d) and (g) because of partially interpreted truth predicate and three-valued inner logic.    

Theory KF, which axiomatizes Kripke's closed-off models (cf. \cite{[F],Fe,[15]}),  don't satisfy norms (d), (g) and (i).   
 Norm (g) is rejected in \cite{Fe}. 
 Preferring norm (g) over any mixture of two different logics, and (h) over partial logic, we use classical logic and construct sublanguages. 
 Modifying Wittgenstein: Whereof one must not speak, thereof one can be silent.

\section{Preliminaries}\label{S3} 

As for the used terminology, cf. e.g.,  \cite{[11],[Ku]}. Denote by $M$  the minimal, countable, transitive and well-founded  model of ZF set theory constructed in \cite{[4]}. It follows from \cite[Lemma II.8.22]{[Ku]} that for each sentence $A$ of the language $L=\{\in\}$  either $A$ or its negation $\neg A$ is
true in $M$ in the sense defined in \cite{[Ku]}  (cf. \cite[II.7 and p. 237]{[Ku]}). 
 $M$ contains by \cite{[4],[13]} natural numbers.
Assume that their names, called numerals, are defined in $L$, e.g., as in  \cite{[Fi]}. 

First we shall  construct a family of sublanguages for the language $\mathcal L=\{\in,T\}$, where $T$ is a monadic predicate. 
Denote by  $D$ the set of
G\"odel numbers of sentences of $\mathcal L$ in its  fixed G\"odel numbering.  Let \#$A$
 denote  the G\"odel number of a sentence $A$ of $\mathcal L$, and $\left\lceil A\right\rceil$ the numeral of \#$A$. 
Given a subset $U$ of  $D$,
define subsets  $G(U)$ and $F(U)$ of $D$ by following rules (iff abbreviates if and only if):   
 \begin{enumerate}
 \item[(r1)] If $A$ is in $L$, then  \#$A$ is in  $G(U)$ iff $A$ is true in $M$, and in $F(U)$ iff $\neg A$ is true in $M$.  
 \item[(r2)]  If $A$ is in $\mathcal L$, then  \#$T(\left\lceil A\right\rceil)$ is in $G(U)$ iff \#$A$  is in $U$, and in $F(U)$ iff  
\#[$\neg A$]  is in $U$.  
\end{enumerate}  
Sentences determined by rules (r1) and (r2), i.e., all sentences $A$ of $L$ and those sentences $T(\left\lceil A\right\rceil)$ of $\mathcal L$ for which \#$A$ or \#[$\neg A$]  is in $U$, are called {\em basic sentences}. 

Next five rules deal with logical connectives. Let $A$ and $B$ be sentences of $\mathcal L$. 
\begin{enumerate}
 \item[(r3)] Negation rule: \#[$\neg A$] is in $G(U)$ iff \#$A$ is in $F(U)$, and in $F(U)$ iff \#$A$ is in $G(U)$. 
 \item[(r4)] Disjunction rule: \#[$A\vee B$]
 is  in $G(U)$ iff  \#$A$ or \#$B$ is in $G(U)$, and in $F(U)$ iff \#$A$ and \#$B$
are in $F(U)$. 
 \item[(r5)] Conjunction rule: \#[$A\wedge B$]
is in $G(U)$ iff \#[$\neg A\vee \neg B$] is in $F(U)$ iff (by (r3) and (r4))  both \#$A$ and \#$B$ are in $G(U)$. Similarly, \#[$A\wedge B$] is in $F(U)$ iff \#[$\neg A\vee \neg B$] is in $G(U)$ iff  \#$A$ or \#$B$ is in $F(U)$.
\item[(r6)] Implication rule: \#[$A\rightarrow B$] is in $G(U)$ iff \#[$\neg A\vee B$] is in $G(U)$ iff (by (r3) and (r4))  \#$A$ is in  $F(U)$ or \#$B$ is in $G(U)$.  \#[$A\rightarrow B$] is  in $F(U)$ iff \#[$\neg A\vee B$] is in $F(U)$ iff \#$A$ is in $G(U)$ and \#$B$ is in $F(U)$.
\item[(r7)] Biconditionality rule:
\#[$A \leftrightarrow B$] is in $G(U)$ iff \#$A$ and \#$B$ are both in $G(U)$ or both in $F(U)$, and  in $F(U)$ iff  \#$A$ is in $G(U)$ and 
\#$B$ is in $F(U)$, or \#$A$ is in $F(U)$ and \#$B$ is in $G(U)$.
\end{enumerate}
When $A(x)$ is a formula in $L$, then quantifications $\exists xA(x)$ and $\forall xA(x)$ are sentences of $L$. Hence rule (r1) is applicable for them.
So it suffices to set rules for $\exists xT(x)$ and $\forall xT(x)$. Assume that the set $X$ of numerals of
G\"odel numbers of sentences of $\mathcal L$ is the intended domain of discourse for $T$. We set the following rules: 
\begin{enumerate}
\item[(r8)] \#[$\exists xT(x)$] is in $G(U)$ iff \#$T(\mathbf n)$ is in $G(U)$ for some $\mathbf n\in X$, and \#[$\exists xT(x)$] is in $F(U)$ iff  \#$T(\mathbf n)$ is in $F(U)$ for every $\mathbf n\in X$.
\item[(r9)] \#[$\forall xT(x)$] is in $G(U)$ iff \#$T(\mathbf n)$ is in $G(U)$ for every  $\mathbf n\in X$, and \#[$\forall xT(x)$] is in $F(U)$ iff \#$T(\mathbf n)$ is in $F(U)$ at least for one $\mathbf n\in X$.
\end{enumerate}
Rules (r1)--(r9) and induction on the complexity of formulas
determine uniquely subsets  $G(U)$ and $F(U)$ of $D$ whenever $U$ is a subset of $D$. 
Sublanguages $\mathcal L_U$ of $\mathcal L$ formed by those sentences $A$ of $\mathcal L$ for which   \#$A$ is in $G(U)$ or in $F(U)$ contain all sentences of $L$ and are closed under logical connectives and quantifiers.

We say that a subset $U$ of  $D$ is  {\em consistent} iff  both \#$A$ and \#[$\neg A$] are not in $U$ for any sentence $A$ of $\mathcal L$. For instance, the empty set $\emptyset$ is consistent.

The following two lemmas have counterparts in \cite{[10]}.

\begin{lemma}\label{L01} Let $U$ be a consistent subset of $D$. Then $G(U)\cap F(U)=\emptyset$,
 i.e., every sentence of $\mathcal L_U$ belongs either to $G(U)$ or to $F(U)$.
\end{lemma}

\begin{proof} If $A$ is in $L$, then  by rule (r1) and \cite[Lemma II.8.22]{[Ku]}, \#$A$ is not in $G(U)\cap F(U)$.

By rule (r2) \#$T(\left\lceil A\right\rceil)$ is in $G(U)$ iff \#$A$ is in $U$, and in $F(U)$ iff 
\#[$\neg A$] is in $U$. Thus \#$T(\left\lceil A\right\rceil)$
cannot be both in $G(U)$ and in $F(U)$ because $U$ is consistent.

Make an  induction hypothesis:
\begin{enumerate}
\item[(h0)] $A$ and $B$ are such sentences of $\mathcal L$ that neither \#$A$ nor \#$B$ is in $G(U)\cap F(U)$. 
\end{enumerate}

As shown above, (h0) holds if $A$ and $B$ are basic sentences.  

If \#[$\neg A$] is in $G(U)\cap F(U)$, then  \#$A$ is in $F(U)\cap G(U)$. Hence, if (h0) holds, then  \#[$\neg A$] is not in $G(U)\cap F(U)$.  

If \#[$A\vee B$] is in $G(U)\cap F(U)$,  then \#$A$ or \#$B$ is in $G(U)$, and both \#$A$ and \#$B$ are in $F(U)$ by (r4), so that \#$A$ or 
\#$B$ is in $G(U)\cap F(U)$.  Hence, if (h0) holds, then  \#[$A\vee B$] is not in $G(U)\cap F(U)$.  

\#[$A\wedge B$] cannot be  in $G(U)\cap F(U)$, for otherwise both \#$A$ and \#$B$ are in $G(U)$, and at least one of \#$A$ and \#$B$ is in $F(U)$, contradicting with (h0).

If  \#[$\neg A$] is in $G(U)\cap F(U)$, then \#$A$ is in $F(U)\cap G(U)$, and (h0) is not valid. Thus, under the hypothesis (h0) neither 
\#[$\neg A$] nor \#$B$ is in $G(U)\cap F(U)$. This result and the above result for disjunction imply that \#[$\neg A\vee B$], or equivalently, 
\#[$A\rightarrow B$], is not in  $G(U)\cap F(U)$. Similarly, \#[$A\leftrightarrow B$] is not in $G(U)\cap F(U)$, for otherwise,  \#$A$ or 
\#$B$ would be in $G(U)\cap F(U)$ by rule (r7), contradicting with (h0).
 
It remains to show that  \#[$\exists xT(x)$] and \#[$\forall xT(x)$]  are not in $G(U)\cap F(U)$. 

If $U$ is empty, then $T(\mathbf  n)$ is by rule (r2) neither in $G(U)$ nor in $F(U)$ for any $\mathbf n\in X$. Thus \#[$\exists xT(x)$] is by rule (r8)  neither in $G(U)$ nor in $F(U)$, and hence not in $G(U)\cap F(U)$. 

If $U$ is not empty, then \#$A$ is in $U$ for some $A$ in $\mathcal L$. Since $U$ is consistent, then \#$[\neg A]$ is not in $U$. By rule (r2), \#$T(\left\lceil A\right\rceil)$ is not in $F(U)$. Thus \#[$\exists xT(x)$] is by rule (r8) not in $F(U)$, and hence not in $G(U)\cap F(U)$.

Because $U$ is consistent, it is a proper subset of $D$. 
Thus there is $n\in D$ such that $n\not\in U$. But $n=\#A$ for some sentence $A$ of $\mathcal L$, whence \#$T(\mathbf n)=\#T(\left\lceil A\right\rceil)$
is not in $G(U)$ by rule (r2). Consequently, \#[$\forall xT(x)$] is by rule (r9) not in $G(U)$, and hence not in $G(U)\cap F(U)$.

The above results and induction on the complexity of formulas imply that \#$A$ is not in $G(U)\cap F(U)$ for any sentence $A$ of $\mathcal L$.  
\end{proof}

\begin{lemma}\label{L201} If $U$ is a consistent subset of $D$, then both $G(U)$ and $F(U$ are consistent.
\end{lemma}

\begin{proof} If $G(U)$ is not consistent, then there is such a sentence $A$ of $\mathcal L$, that \#$A$ and \#[$\neg A$] are in $G(U)$. Because \#[$\neg A$] is in $G(U)$, then \#$A$ is also in $F(U)$ by rule (r3), and hence in $G(U)\cap F(U)$. But then, by Lemma \ref{L01}, $U$ is not consistent.
Consequently, if $U$ is  consistent, then $G(U)$ is consistent. The proof that $F(U)$ is consistent if $U$ is, is similar.
\end{proof}

\section{A mathematical theory of truth and its properties}\label{S4}

Recall that  
$D$ denotes the set of G\"odel numbers of sentences of the language $\mathcal L=\{\in,T\}$. Given  a subset $U$ of  $D$, let $G(U)$ and $F(U)$ be the subsets  of $D$ constructed in Section \ref{S3}.
  
\begin{definition}\label{D41} {\it Assume that $U$ is a consistent subset of $D$, and that $U=G(U)$.  
Denote by $\mathcal L_U$ the language containing those sentences $A$ of $\mathcal L=\{\in,T\}$ for which   \#$A$ is in $G(U)$ or in $F(U)$.
A  sentence $A$ of $\mathcal L_U$ is interpreted as true iff \#$A$ is in $G(U)$, and as false iff \#$A$ is in $F(U)$.
$T$ is called a truth predicate for $\mathcal L_U$}.  \end{definition}

The existence of consistent subsets $U$ of $D$ satisfying $U=G(U)$, including the smallest one, is proved in Section \ref{S5}.
\smallskip

The main goal of this paper is to prove the following Theorem for the above defined mathematical theory of truth (shortly MTT).

\begin{theorem}\label{T31} Assume that $U$ is a consistent subset of $D$, and that $U=G(U)$. Then MTT satisfies all the norms (a)--(j) presented in the Introduction.
\end{theorem}

\begin{proof} (a) The syntax of MTT is comprised by logical symbols of first-order predicate logic, nonlogical symbols $\in$ and $T$, natural numbers as constants, numerals as terms, and variables ranging in $M$.   $T$ is the truth predicate for $\mathcal L_U$, by Definition \ref{D41}.

(b)  Assume that $A$ is a sentence of $L$. By \cite[Lemma II.8.22]{[Ku]}  either $A$  is true in $M$, or $\neg A$ is true in $M$ ($A$ is false in $M$). 

-- $A$ is true in $M$ iff \#$A$ is in $G(U)$, by rule (r1), iff \#$A$ is in $U$, because $U=G(U)$, iff  \#$T(\left\lceil A\right\rceil)$ is in $G(U)$ by rule (r2), iff $T(\left\lceil A\right\rceil)$ is true,  by Definition \ref{D41}.

-- $\neg A$ is true in $M$   iff  \#[$\neg A$] is in $G(U)$, by rule (r1), iff \#[$\neg A$] is in $U$, because $U=G(U)$,  iff  \#$\left\lceil T(A)\right\rceil$ is in $F(U)$, by rule (r2), iff $T(\left\lceil A\right\rceil)$ is false, by Definition \ref{D41}.   
   
Consequently, a sentence $A$ of $L$ is true in $M$ iff $T(\left\lceil A\right\rceil)$ is true, and false in $M$ iff $T(\left\lceil A\right\rceil)$ is false.
Thus MTT proves the $ZF$ set theory modelled by $M$ true.

(c) $T$ is by Definition \ref{D41} a truth predicate for $\mathcal L_U$, and is not subject to any restrictions within $\mathcal L_U$.

(d)  Let $A$ be a sentence of $\mathcal L_U$.  Since $U$ is consistent, then \#$A$ is by Lemma \ref{L01} either in $G(U)$ or in $F(U)$.  
Applying rules (r2) and (r3), and the assumption  $U=G(U)$, we obtain
\newline
-- \#$A$ is in $G(U)$ iff \#$A$ is in $U$  iff 
\#$T(\left\lceil A\right\rceil)$ is in $G(U)$;
\newline
-- \#$A$ is in $F(U)$ iff \#[$\neg A$] is in $G(U)$ iff \#[$\neg A$] is in $U$  iff \#$T(\left\lceil A\right\rceil)$ is in $F(U)$.  
\newline
Consequently, \#$A$ and \#$T(\left\lceil A\right\rceil)$ are both either in $G(U)$ or in $F(U)$. Thus \#[$A\leftrightarrow T(\left\lceil A\right\rceil)$] is by rule (r7) in $G(U)$, so that $A\leftrightarrow T(\left\lceil A\right\rceil)$ is true  by Definition \ref{D41}. This holds for every sentence $A$ of $\mathcal L_U$, whence $T$-biconditionals $A\leftrightarrow T(\left\lceil A\right\rceil)$ are derivable unrestrictedly in $\mathcal L_U$.

(e) It follows from Definition \ref{D41} that `\#$A$ is in $G(U)$' can be replaced by `$A$ is true' and `\#$A$ is in $F(U)$' by 
`$A$ is false'. 
Thus rules (r3)--(r9) imply that the truth in MTT is compositional.

(f) MTT allows for standard interpretations. $M$ is by \cite{[13]} a standard model of set theory. 

(g)  The outer logic and the inner logic coincide in MTT because of (d), as stated \cite{[16]}. 

(h)  By the proof of (e) the logic of $\mathcal L_U$, i.e., the inner logic of MTT is classical. This and (g) imply that also the outer logic of MTT is classical. 

(i) Since $U$ is consistent, it follows from Lemmas \ref{L01} and \ref{L201} that every sentence of $\mathcal L_U$ is either true or false in the sense of Definition \ref{D41}. Thus MTT is free from paradoxes. Only mathematical methods are used in it's formulation, whence it is a mathematical theory. 

(j) By Definition \ref{D41} truth for the language is defined within the language $\mathcal L_U$ itself.
\end{proof}
 
\begin{remark}\label{R41} If $\theta$ is a monadic predicate of $L$, there is such a sentence $A$ in $L$ that biconditionality $A\leftrightarrow\neg\theta(\left\lceil A\right\rceil)$ is provable from axioms of ZF set theory (cf., e.g., \cite[Lemma IV.5.31]{[Ku]}). 
Thus $T$-biconditionality rule (d) does not hold in $L$ when $T=\theta$.  In particular, $L$ does not contain its truth predicate.
This holds also for the language $\mathcal L$.
To satisfy all norms (a)--(j) it is then essential that those languages $\mathcal L_U$  for which MTT is formulated differ both from $L$ and from $\mathcal L$.
This holds in our approach, since  $G(U)\cup F(U)$ is a proper and nonempty subset of $D$ when $U$ is a consistent fixed point of $G$.

The language $\mathcal L_U$ corresponding to the smallest consistent fixed point $U$ of $G$ relates to the language of grounded sentences defined in \cite{[10],[15]}.
See also \cite{[6]}, where considerations are restricted to signed statements of $\mathcal L$.

MTT is  only a {\em theory} of truth depending, e.g., on the semantic notion 'true in $M$'. Moreover, the construction of $M$ starts from the set $\omega\cup\{\omega\}$, the existence of which is a matter of faith. 


\end{remark}  


\section{Appendix}\label{S5}

Let $\mathcal P$ denote the family of all consistent subsets of the set $D$ of 
G\"odel numbers of sentences of $\mathcal L$. Fixed points of  the set mapping $G:=U\mapsto G(U)$ from $\mathcal P$ to $\mathcal P$, i.e., those $U\in \mathcal P$ for which $U=G(U)$, have a central role in the formulation of MTT.  Before the proof of our main fixed point theorem we prove some 
auxiliary results.  

\begin{lemma}\label{L202} Let $U$ and $V$ be sets of $\mathcal P$, and assume that  $U \subseteq V$. If $A$ is a sentence of $\mathcal L$, then 
\#$A$ is  in $G(V)$ whenever it is in $G(U)$,
and \#$A$ is in $F(V)$ whenever it is in $F(U)$. 
\end{lemma}

\begin{proof} Assume that $U\subseteq V$. Consider first basic sentences. Let $A$ be a sentence of $L$. By rule (r1) \#$A$ is in $G(U)$ and also in $G(V)$ iff $A$ is true in $M$, and  both in $F(U)$ and in $F(V)$ iff $\neg A$ is true in $M$.
  
Let $A$ be a sentence of $\mathcal L$. If \#$T(\left\lceil A\right\rceil)$ is in $G(U)$, then \#$A$ is in $U$ by rule (r2). Because $U\subseteq V$, then \#$A$ belongs to $V$, whence \#$T(\left\lceil A\right\rceil)$ is in $G(V)$ by rule (r2).

If  \#$T(\left\lceil A\right\rceil)$ is in $F(U)$, then \#[$\neg A$] is in $U$ by rule (r2). Since $U\subseteq V$, then \#[$\neg A$] belongs to $V$, so that \#$T(\left\lceil A\right\rceil)$ is in $F(V)$ by rule (r2).

Thus all basic sentences satisfy the lemma.

Assume  that $A$ is a sentence of $\mathcal L$. If \#[$\neg A$] is in $G(U)$ but not in $G(V)$, then \#$A$ is in $F(U)$ but not in $F(V)$ by rule (r3). If \#[$\neg A$] is in $F(U)$ but not in $F(V)$, then \#$A$ is in $G(U)$ but not in $G(V)$ by rule (r3). Hence, if $A$ satisfies the lemma, then also $\neg A$ satisfies it.

Make an induction hypothesis:
\begin{enumerate}
\item[(h1)] $A$ and $B$ are such sentences of $\mathcal L$ that \#$A$ is in $G(V)$ if it is in $G(U)$,  \#$A$ is in $F(V)$ if it is in $F(U)$, 
\#$B$ is in $G(V)$ if it is in $G(U)$, and \#$B$ is in $F(V)$ if it is in $F(U)$. 
\end{enumerate}

If \#[$A\vee B$] is in $G(U)$, then \#$A$ or \#$B$ is in $G(U)$ by rule (r4). By (h1) \#$A$ or \#$B$ is in $G(V)$, so that \#[$A\vee B$] is in $G(V)$. If \#[$A\vee B$] is in $F(U)$, then  \#$A$ and \#$B$ are in $F(U)$ by rule (r4), and hence also in $F(V)$, by (h1), so that \#[$A\vee B$] is in $F(V)$.
Thus $A\vee B$ satisfies the lemma if (h1) holds.

If \#[$A\wedge B$] is in $G(U)$, then both \#$A$ and \#$B$ are in $G(U)$ by rule (r5), and hence also in $G(V)$, by (h1). Thus \#[$A\wedge B$] is in $G(V)$. If \#[$A\wedge B$] is in $F(U)$, then  \#$A$ or \#$B$ is in $F(U)$ by rule (r5), and hence also in $F(V)$, by (h1), whence \#[$A\wedge B$] is in $F(V)$.
Thus $A\wedge B$ satisfies the lemma if (h1) holds.

If \#[$A\rightarrow B$] is in $G(U)$, then \#[$\neg A$] or \#$B$ is in $G(U)$,  i.e., \#$A$ is in $F(U)$ or \#$B$ is in $G(U)$. Then, by (h1),  \#$A$ is in $F(V)$ or \#$B$ is in $G(V)$, i.e., \#[$\neg A$] or \#$B$ is in $G(V)$. Thus  \#[$A\rightarrow B$] is in $G(V)$. If \#[$A\rightarrow B$] is in $F(U)$, then \#$A$ is in $G(U)$ and \#$B$ is in $F(U)$. This implies by (h1) that \#$A$ is in $G(V)$ and \#$B$ is in $F(V)$, so that \#[$A\rightarrow B$] is in $F(V)$. 
Thus $A\rightarrow B$ satisfies the lemma if (h1) holds.

If \#[$A\leftrightarrow B$] is in $G(U)$, then \#$A$ and \#$B$ are both in $G(U)$ or both in $F(U]$, and hence both in $G(V)$ or both in $F(V)$ if (h1) holds, so that \#[$A\leftrightarrow B$] is in $G(V)$.
If \#[$A\leftrightarrow B$] is in $F(U)$, then \#$A$ is in $G(U)$ and \#$B$ is in $F(U)$ or vice versa. Thus \#$A$ is in $G(V)$ and 
\#$B$ is in $F(V)$ or vice versa  if (h1) holds, so that \#[$A\leftrightarrow B$] is in $F(V)$. Thus $A\leftrightarrow B$ satisfies the lemma if (h1) holds.
Because (h1)  holds for basic sentences, the above results and induction on the complexity of expressions imply that logical connectives of sentences of $\mathcal L$ satisfy the lemma. 

It follows from the proof of Lemma \ref{L01} that  \#[$\exists xT(x)]$ is neither in $F(U)$ nor in $F(V)$.  
If  \#[$\exists xT(x)$] is  in $G(U)$, then $T(\mathbf n)$ is in $G(U)$ for some $\mathbf n\in X$. Since $\mathbf n=\left\lceil A\right\rceil$ for some sentence $A$ of $\mathcal L$, then \#$T(\left\lceil A\right\rceil)$ is in $G(U)$. Thus \#$A$ is by rule (r2) in $U$, and hence also in $V$. Then, by rule (r2), \#$T(\left\lceil A\right\rceil)=\#T(\mathbf n)$ is in $G(V)$, so that \#[$\exists xT(x)]$ is in $G(V)$ by rule (r8). Consequently,   
\#[$\exists xT(x)$] is in $G(V)$ whenever it is in $G(U)$.

The proof of Lemma \ref{L01} implies  that \#[$\forall xT(x)$] is  neither in  $G(U)$ nor in $G(V)$. If \#[$\forall xT(x)$] is in $F(U)$, then $T(\mathbf n)$ is in $F(U)$ for some $\mathbf n\in X$. But $\mathbf n=\left\lceil A\right\rceil$ for some sentence $A$ of $\mathcal L$, so that \#$T(\left\lceil A\right\rceil)$ is in $F(U)$. So \#[$\neg A$] is by rule (r2) in $U$, and hence also in $V$. Thus, by rule (r2), \#$T(\left\lceil A\right\rceil)=\#T(\mathbf n)$ is in $F(V)$, so that \#[$\exists xT(x)]$ is in $F(V)$ by rule (r8). Consequently, \#[$\forall xT(x)$] 
is in $F(V)$ whenever it is in $F(U)$.
 
The above results  imply the conclusion of the lemma. 
\end{proof}

According to Lemma \ref{L201} the mapping $G:=U\mapsto G(U)$ maps $\mathcal P$ into $\mathcal P$.
Assuming  that $\mathcal P$ is ordered by inclusion, the above lemma implies the following result.

\begin{lemma}\label{L203}  $G$ is order preserving in $\mathcal P$, i.e., 
 $G(U)\subseteq G(V)$ whenever  $U$ and $V$ are  sets of  $\mathcal P$ and $U\subseteq V$.
\end{lemma}

As above lemmas, also the next lemma has a counterpart in \cite{[10]}.

\begin{lemma}\label{L204} 
(a) If $\mathcal W$ is a chain in $\mathcal P$, then the union   
$\cup \mathcal W=\cup\{U\mid U\in \mathcal W\}$ is  consistent.

(b) The intersection  $\cap \mathcal W=\cap\{U\mid U\in \mathcal W\}$ of every nonempty subfamily $\mathcal W$ of $\mathcal P$ is a  consistent subset of $D$.
\end{lemma}

\begin{proof} (a) Assume on the contrary that $\cup \mathcal W$ is not consistent. Then there is a such a sentence $A$ of $\mathcal L$
that  both \#$A$ and \#[$\neg A$] are in $\cup \mathcal W$. Thus
$\mathcal W$ has a member, say $U$, which contains \#$A$,
and a member, say $V$, which contains \#[$\neg A$]. If $\mathcal W$ is a chain, then $U\subseteq V$ or $V\subseteq U$. In former case $V$ and in latter case 
$U$ contains both \#$A$ and \#[$\neg A$]. But this is impossible because  $\mathcal W$ is a subfamily of $\mathcal P$. This proves (a).

(b) The intersection  $\cap \mathcal W$ is a subset of $D$, and is contained in every member of $\mathcal W$.  Hence $\cap \mathcal W$
is consistent, for otherwise there is such a sentence $A$ in $\mathcal L$ that both \#$A$ and \#[$\neg A$] are in $\cap \mathcal W$. Then every member of $\mathcal W$ would also contain both \#$A$ and \#[$\neg A$]. But this is impossible because every member of 
$\mathcal W$ is consistent. This proves (b).
\end{proof}

In the formulation and the proof of our main fixed point theorem we use transfinite sequences of $\mathcal P$ indexed by von Neumann ordinals. Such a sequence $(U_\lambda)_{\lambda\in\alpha}$ of $\mathcal P$ is said to be strictly increasing if 
$U_\mu\subset U_\nu$ whenever $\mu\in\nu\in\alpha$, and strictly decreasing if 
$U_\nu\subset U_\mu$ whenever $\mu\in\nu\in\alpha$. 
A set $V$ of $\mathcal P$ is called  {\em sound} iff $V\subseteq G(V)$. 

The following fixed point theorem is an application of Lemmas \ref{L201}, \ref{L203} and \ref{L204}.

\begin{theorem}\label{T1} 
If $V\in \mathcal P$ is  sound, then there exists the smallest of those  fixed points of $G$  which contain $V$. This fixed point is the last member of the union of those transfinite sequences 
$(U_\lambda)_{\lambda\in\alpha}$ of $\mathcal P$ which satisfy
\begin{itemize}
\item[(C)] $(U_\lambda)_{\lambda\in\alpha}$ is strictly increasing,  
$U_0=V$, and if $0\in\mu\in \alpha$, then
$U_\mu = \underset{\lambda\in\mu}{\bigcup}G(U_\lambda)$.  
\end{itemize}  
\end{theorem}

\begin{proof} Let $V\in\mathcal P$ be sound. 
Transfinite sequences of $\mathcal P$  satisfying (C) are called  $G$-sequences. 
We shall first show that $G$-sequences are nested:
\begin{enumerate}
\item[(1)] {\it Assume that $(U_\lambda)_{\lambda\in\alpha}$ and $(V_\lambda)_{\lambda\in\beta}$ are $G$-sequences, and that $\{U_\lambda\}_{\lambda\in\alpha} \not\subseteq \{V_\lambda\}_{\lambda\in\beta}$.
Then $(V_\lambda)_{\lambda\in\beta}  = (U_\lambda)_{\lambda\in\beta}$}.
\end{enumerate}
By the assumption of (1)  $\mu = \min \{\lambda\in\alpha\mid U_\lambda\not\in\{V_\lambda\}_{\lambda\in\beta}\}$ exists, and 
$\{U_\lambda\}_{\lambda\in\mu} \subseteq \{V_\lambda\}_{\lambda\in\beta}$.  
Properties (C) imply by transfinite induction that $U_\lambda=V_\lambda$ for each $\lambda\in\mu$. To prove that $\mu=\beta$,
 make a counter-hypothesis: $\mu\in\beta$. Since $\mu\in\alpha$ and $U_\lambda=V_\lambda$ for each $\lambda\in\mu$, it follows from properties (C) that   $U_\mu = \underset{\lambda\in\mu}{\bigcup}G(U_\lambda) = \underset{\lambda\in\mu}{\bigcup}G(V_\lambda)=V_\mu$, which is
impossible, since $V_\mu\in \{V_\lambda\}_{\lambda\in\beta}$, but $U_\mu\not\in \{V_\lambda\}_{\lambda\in\beta}$. Consequently, $\mu=\beta$ and $U_\lambda=V_\lambda$ for each $\lambda\in\beta$, whence
 $(V_\lambda)_{\lambda\in\beta} =(U_\lambda)_{\lambda\in\beta}$.

By definition, every $G$-sequence $(U_\lambda)_{\lambda\in\alpha}$ is a function $\lambda\mapsto U_\lambda$ from $\alpha$ into $\mathcal P$. Property (1) implies that these functions are compatible. Thus their union is by \cite[Theorem 2.3.12]{[12]} a function with values in $\mathcal P$, the domain being the union of all index sets  of $G$-sequences. Because these index sets are ordinals, then their union is also an ordinal by \cite[I.8.10]{[Ku]}. Denote it by $\gamma$. The union function can be represented as a sequence $(U_\lambda)_{\lambda\in\gamma}$ of $\mathcal P$. It is strictly increasing as a union of strictly increasing nested sequences.  

To show that $\gamma$ is a successor, assume on the contrary that $\gamma$ is a limit ordinal. Given $\nu\in\gamma$, then $\mu=\nu\cup\{\nu\}$ and $\alpha=\mu\cup\{\mu\}$ are in $\gamma$, and $(U_\lambda)_{\lambda\in\alpha}$ is a $G$-sequence. Denote $U_\gamma = \underset{\lambda\in\gamma}{\bigcup}G(U_\lambda)$.
 $G$ is order preserving  by Lemma \ref{L203}, and $(U_\lambda)_{\lambda\in\gamma}$ is a strictly increasing sequence of $\mathcal P$. Thus  $\{G(U_\lambda)\}_{\lambda\in\gamma}$ is a chain in  $\mathcal P$, whence $U_\gamma$ is consistent by Lemma \ref{L204}(a). Moreover, 
 $U_\nu\subset U_\mu=\underset{\lambda\in\mu}{\bigcup}G(U_\lambda)\subseteq U_\gamma$. This holds for each $\nu\in\gamma$, whence $(U_\lambda)_{\lambda\in\gamma\cup\{\gamma\}}$ is a $G$-sequence. This is impossible, since $(U_\lambda)_{\lambda\in\gamma}$ is the union of all $G$-sequences.
Consequently, $\gamma$ is a successor, say $\gamma=\alpha\cup\{\alpha\}$. Thus $U_\alpha$ is the last member of $(U_\lambda)_{\lambda\in\gamma}$,
$U_\alpha=\max\{U_\lambda\}_{\lambda\in\gamma}$, and $G(U_\alpha)=\max\{G(U_\lambda)\}_{\lambda\in\gamma}$. Moreover, $(U_\lambda)_{\lambda\in\gamma}$ is a $G$-sequence, for otherwise $(U_\lambda)_{\lambda\in\alpha}$ would be the union of all $G$-sequences. In particular,
$U_\alpha=\underset{\lambda\in\alpha}{\bigcup}G(U_\lambda)\subseteq \underset{\lambda\in\gamma}{\bigcup}G(U_\lambda)=G(U_\alpha)$,
so that $U_\alpha\subseteq G(U_\alpha)$. Equality holds, since otherwise the longest $G$-sequence $(U_\lambda)_{\lambda\in\gamma}$ could be extended by  $U_\gamma= \underset{\lambda\in\gamma}{\bigcup}G(U_\lambda)$.
Thus  $U_\alpha$ is a fixed point of $G$ in $\mathcal P$.  

Assume that $W\in\mathcal P$ is a fixed point of $G$, and that $V\subseteq W$. Then $U_0=V\subseteq W$. If $0\in\mu\in\gamma$, and $U_\lambda\subseteq W$ for each $\lambda\in\mu$, then $G(U_\lambda)\subseteq G(W)$ for each $\lambda\in\mu$, whence $U_\mu=\underset{\lambda\in\mu}{\bigcup}G(U_\lambda)\subseteq G(W)=W$. Thus, by transfinite induction, $U_\mu\subseteq W$ for each $\mu\in\gamma$. Thus $U_\alpha\subseteq W$ so that $U_\alpha$ is the smallest  fixed point of $G$ that contains $V$.
\end{proof}

As a consequence of Theorem \ref{T1} we obtain.

\begin{corollary}\label{C1} $G$ has the smallest  fixed point.
\end{corollary}

\begin{proof}  The empty set $\emptyset$ is both sound and consistent.  Thus the smallest  fixed point of $G$  that contains $\emptyset$ exists by  Theorem \ref{T1}. It is the smallest  fixed point of $G$, since every fixed point of $G$ contains $\emptyset$.  
\end{proof}

Next we prove  that every consistent subset of $D$ has the greatest sound and consistent subset.

\begin{theorem}\label{T2}  The equation $V=W\cap G(V)$ has for each consistent subset $W$  of $D$ the greatest solution $V$ in $\mathcal P$.  It is the greatest sound set of $\mathcal P$ that is contained in $W$. $V$ is the last member of
the union of those transfinite sequences 
$(V_\lambda)_{\lambda\in\alpha}$ of $\mathcal P$ which satisfy
\begin{itemize}
\item[(D)] $(V_\lambda)_{\lambda\in\alpha}$ is strictly decreasing,  
$V_0=W$, and if $0\in\mu\in \alpha$, then
$V_\mu = W\cap(\underset{\lambda\in\mu}{\bigcap}G(V_\lambda))$.  
\end{itemize}  
\end{theorem}

\begin{proof} Assume that $W$ is a consistent subset of $D$. Like the proof that the union of all $G$-sequences is a  $G$-sequence one can prove that the union of all transfinite sequences which have properties (D) given in Theorem \ref{T2} has property (D). Let  $(V_\lambda)_{\lambda\in\gamma}$
 be that sequence.
Denote $V=W\cap(\underset{\lambda\in\gamma}{\bigcap}G(V_\lambda))$. Because $W$ and the sets $V_\lambda$, $\lambda\in\gamma$, are consistent, it follows from Lemma \ref{L201} and Lemma \ref{L204}(b) that $V$ is consistent. Moreover, $V\subseteq V_\lambda$ for each $\lambda\in\gamma$.
If $V\subset V_\lambda$ for each $\lambda\in\gamma$, then the choice  $V_\lambda=V$ implies that $(V_\lambda)_{\lambda\in\gamma\cup\{\gamma\}}$ satisfies (D) when $\alpha=\gamma\cup\{\gamma\}$. But this is impossible because of the choice of  $(V_\lambda)_{\lambda\in\gamma}$. Thus 
$V = \min\{V_\lambda\}_{\lambda\in\gamma}$, and $V$ is the last member of $(V_\lambda)_{\lambda\in\gamma}$ because this sequence is strictly decreasing. Since $G$ is order preserving, then $G(V) = \min\{G(V_\lambda)\}_{\lambda\in\gamma}=\underset{\lambda\in\gamma}{\bigcap}G(V_\lambda)$.
Thus $V = W\cap G(V)$, so that $V\subseteq G(V)$, i.e., $V$ is sound and is contained in $W$.

Assume that $U$ is consistent, that $U\subseteq G(U)$, and that $U\subseteq W$. Since $V_0=W$ by (D$_\gamma$), then $U\subseteq V_0$. If $0\in\mu\in\gamma$ and $U\subseteq V_\lambda$ for each $\lambda\in\mu$, then  $G(U)\subseteq G(V_\lambda)$ for each $\lambda\in\mu$, whence
$U\subseteq W\cap G(U)\subseteq W\cap(\underset{\lambda\in\mu}{\bigcap}G(V_\lambda))=V_\mu$. Thus, by transfinite induction, $U\subseteq V_\lambda$ for each $\lambda\in\gamma$, so that $U\subseteq \min\{V_\lambda\}_{\lambda\in\gamma}=V$. Consequently, $V$ is the greatest sound and consistent subset of $D$ that is contained in $W$.
\end{proof}

The following results shows that if $W$ is any consistent set of G\"odel numbers of sentences of $L$, and $V$ is determined by Theorem \ref{T2}, then
the union of the $G$- sequences (C) of Theorem \ref{T1} is the smallest fixed point of $G$.

\begin{corollary}\label{C2} Let $W$ be a consistent subset of $D$ containing only  G\"odel numbers of sentences of $L$. Then

(a) $V= W\cap G(W)$ is the greatest consistent solution of $V=V\cap G(V)$  contained in $W$.

(b) The union of the $G$- sequences (C) of $V= W\cap G(W)$ is the smallest fixed point of $G$.
\end{corollary} 

\begin{proof} (a) Let $W$ be a consistent set of  G\"odel numbers of sentences of $L$.  $G(W)$ contains by rule (r1) G\"odel numbers of those  sentences of $L$ which are true in $M$, and only those. Thus $V= W\cap G(W)$ is the set of  G\"odel numbers of those sentences of $W$ which are true in $M$. Hence $V\subset G(\emptyset)\subseteq G(V)$, so that $V$ is sound. It is also consistent, as an intersection of two consistent sets. Consequently, $V= W\cap G(W)$ is a solution of $V=V\cap G(V)$, and by Theorem \ref{T2} the greatest solution contained in $W$.  
\smallskip

(b) $V=W\cap G(W)$ is by (a) sound and consistent, and  $V\subset G(\emptyset)$. If $U$ is a fixed point of $G$, then $V\subset G(\emptyset)\subset G(U)=U$. Thus $V$ is contained in the smallest fixed point of $G$. By Theorem \ref{T1}, the union of the $G$- sequences (C) is the smallest fixed point of $G$ that contains $V$. This proves (b).
\end{proof}

\begin{remark}\label{R51}
The smallest members of  $(U_\lambda)_{\lambda\in\alpha}$ satisfying (C) are $n$-fold iterations $U_n=G^n(V)$, $n\in\mathbb N=\{0,1,\dots\}$. If they form a strictly increasing sequence, the next member $U_\omega$ is their union, $U_{\omega+n}= G^n(U_\omega)$, $n\in\mathbb N$, and so on.
\smallskip

If the set $W$ is finite, then
the longest sequence $(V_\lambda)_{\lambda\in\alpha}$ satisfying (D) is obtained by the finite algorithm:  
 $V_0=W$. For $n$ from $0$ while $V_n\ne W\cap G(V_n)$ do: $V_{n+1}=W\cap G(V_n)$.

Zorn's Lemma, together with Lemmas \ref{L203} and \ref{L204}, can be applied, e.g., as in  \cite{[6],[10]}, to prove results of Theorem \ref{T1} and Corollary \ref{C1}, and also the existence of maximal fixed points of $G$. ZF set theory is no more sufficient framework for such proofs, since Zorn's Lemma is equivalent to the Axiom of Choice. 
The result of Theorem \ref{T2} that gives a method to determine sound subsets of $D$ seems to be new.

As for generalizations of Theorems \ref{T1} and \ref{T2}, see, e.g., \cite{[3],[9]}. 

\end{remark}

\smallskip

{\bf Acknowledgements:} The author is indebted to Prof. Hannes Leitgeb and to Ph.D. Tapani Hyttinen for valuable discussions on the subject. The present work is influenced by  \cite{[10]}.


\smallskip

\baselineskip12pt
\begin{thebibliography}{99}
\scriptsize
\smallskip



\bibitem{[BG]} { Beall, Jc \& Glandzberg, Michael} (2013) {Liar Paradox}, The Stanford Encyclopedia of Philosophy, Edward N. Zalta (ed.).
 Routledge.


\bibitem{[3]} {   Carl, S. \& Heikkil\"a, S.} (2011) { Fixed Point Theory in Ordered Sets and Applications}, {
 Springer, Berlin Heidelberg New York}. 

\bibitem{[C]} {  Chomsky, Noam} (1957)  Syntactic structures, {The Hague: Mouton}.

\bibitem{[4]} {  Cohen, Paul} (1963)  A minimal model for set theory, {
 Bulletin of the American Mathematical
Society}, {\bfseries 69}, 537--540.


\bibitem{[F]} {  Feferman, Solomon} (1984) {Toward useful type-free theories}, {Journal of Symbolic Logic}, {\bfseries 49}, 1, 75--111.  

\bibitem{Fe} {  Feferman, Solomon} (2012) Axiomatizing truth. Why and how?  Logic, Construction, Computation (U. Berger et al. eds.)
Ontos Verlag, Frankfurt, 185–-200.

\bibitem{[6]} {  Fitting,  Melvin } (1986)
{ Notes on the Mathematical Aspects of Kripke's Theory of Truth}, {
 Notre Dame Journal of Formal Logic}, {\bfseries 27},1, 75--88.

\bibitem{[Fi]} {  Fitting,  Melvin } (2007), Incompleteness in the Land of Sets, Studies in Logic. 



\bibitem{[H]} {   Halbach, Volker } (1997) { Tarskian and Kripkean truth}, {
Journal of Philosophical Logic}, {\bfseries 26}, 1,  69--80.


\bibitem{[Ha]} {   Halbach, Volker } (2011) { Axiomatic Theories of Truth}, {
 Cambridge University Press}.


\bibitem{[8]} {   Halbach, Volker \& Horsten, Leon } (2006) {Axiomatizing Kripke's Theory of Truth}, {
 J, Symbolic Logic} {\bfseries 71},2, 677--712.


\bibitem{[10]} {   Heck, Richard G. }, { Kripke's theory of truth}, {
 http:/www.frege.brown.edu \newline
/heck/philosophy/pdf/notes
/KripkesTheoryOfTruth.pdf}.

\bibitem{[11]} {    Heck, Richard G.} { Formal background for theories of truth}, \newline{
 http:/www.   frege.brown.edu/heck  /philosophy
 /pdf /notes/FormalBackground.pdf}.

\bibitem{[9]} Heikkil\"a, S. (2014) A Mathematical Theory of Truth with Applications, arXiv.1307.4692.


\bibitem{Ho} { Hodges, Wilfried} (2013) {Tarski's Truth Definitions}, The Stanford Encyclopedia of Philosophy, Edward N. Zalta (ed.).

\bibitem{[12]} {   Hrbacek, Karel \& Jech, Thomas } (1999) { Introduction to Set Theory}, {
 Marcel Dekker, Inc, New York - Basel}.

\bibitem{[13]} {   Kanamori,  Akihiro } (2008)  {Cohen and Set Theory},
 {The Bulletin of Symbolic Logic}, {\bfseries 14},3, 351--377. 



\bibitem{[15]} {   Kripke, Saul } (1975) {Outline of a Theory of Truth}, {
 Journal of Philosophy}, {\bfseries  72}, 690--716.

\bibitem{[Ku]} {Kunen, Kenneth} (2009) The Foundations of Mathematics, Mathematical Logic and Foundations, {\bfseries 19}.

\bibitem{[16]} {  Leitgeb, Hannes } (2007) { What Theories of Truth Should be Like (but Cannot be)}, 
 {
   Philosophy Compass}, {\bfseries 2/2}, 276--290.

\bibitem{[MG]} {McGee, Vann} (1989) { Applying Kripke's Theory of Truth}, {The Journal of Philosophy}, {\bfseries 86},10, 530--539. 


\bibitem{[19]} Tarski, A (1955) A lattice-theoretical fixpoint theorem and its applications, Pacific T. Math. {\bfseries 5}, 285--309. 

 \bibitem{[20]} Tarski, Alfred  (1936) {Der Wahrheitsbegriff in Formalisierten Sprachen}, Studia Philosophica 1: 261--405.

 



\end{thebibliography}
\end{document}